\documentclass{amsart}
\usepackage{vmargin,amsmath,amsfonts,amssymb,amsthm,xcolor,amscd,latexsym,xspace,enumerate}
\usepackage[pagebackref, colorlinks=true,linkcolor=blue,citecolor=red]{hyperref}
\usepackage{hyperref}
\usepackage{pdfpages}
\usepackage{tikz-cd}
\usetikzlibrary{arrows}
\usetikzlibrary{intersections}
\usepackage{pgfplots}
\usetikzlibrary{calc,3d,shapes, pgfplots.external, intersections}
\usepackage{pgfplots}

\newtheorem{theorem}{Theorem}[section]
\newtheorem{corollary}[theorem]{Corollary}
\newtheorem{lemma}[theorem]{Lemma}

\theoremstyle{definition}
\newtheorem{definition}[theorem]{Definition}
\newtheorem{remark}[theorem]{Remark}

\newtheorem{example}[theorem]{Example}
\numberwithin{equation}{section}
\usepackage{stmaryrd}
\usepackage{mathtools}
\usepackage{xcolor, soul}
\sethlcolor{green}

\title[On locally compact groups of small topological entropy]{On locally compact groups of small topological entropy}

\author[F. G. Russo]{Francesco G. Russo$^\natural$ }
\address{  Francesco G. Russo and Olwethu Waka\endgraf
Department of Mathematics and Applied Mathematics\endgraf 
University of Cape Town\endgraf
Private Bag X1, Rondebosch 7701\endgraf 
Cape Town, South Africa\endgraf
Emails: \texttt{francescog.russo@yahoo.com}\endgraf
\hspace{1.2cm}\texttt{wolwethu@gmail.com}
}

\author[O. Waka]{Olwethu Waka}

\date{4th of April 2023 \endgraf
$^\natural$ Corresponding author}
\begin{document}

\maketitle

\begin{abstract} We discuss the finiteness of the topological entropy  of continuous endomorphims for some classes of locally compact groups. Firstly, we focus on the  abelian case, imposing the condition of being compactly generated, and note an interesting behaviour of  slender  groups. Secondly, we remove the condition of being abelian and consider nilpotent periodic locally compact  $p$-groups ($p$ prime), reducing the computations  to the case of Sylow $p$-subgroups. Finally, we investigate locally compact Heisenberg $p$-groups $\mathbb{H}_{n}(\mathbb{Q}_{p})$ on the field $\mathbb{Q}_{p}$ of the $p$-adic rationals with $n$ arbitrary positive integer. 
\end{abstract}

%\tableofcontents

\section{Motivations  and  Main Results}

In the present paper a locally compact group is always assumed to be a topological group whose topology is both Hausdorff and locally compact. Hood   \cite{H} formulated a notion of topological entropy involving the well known concept of uniformity for a topological space. His definition  applies to a topological groups possessing a left uniformity, since continuous endomorphisms  are uniformly continuous (in connection with the given left uniformity). Let's be more formal on Hood's Entropy \cite{H} in the context of what we need to investigate here. For a   locally compact group $G$, we denote by $\mathcal{CT}(G)$ the collection of all compact neighborhoods of the identity of $G$, and by $\mu$ a left invariant Haar measure on $G$. For a continuous endomorphism $\varphi$ of $G$, an element $V\in\mathcal{CT}(G)$ and an $n\in\mathbb{N}=\{1, 2, 3, \ldots\}$,  \begin{equation}C_n(\varphi,V)=V\cap\varphi^{-1}(V)\cap\ldots\cap\varphi^{-n+1}(V)\in\mathcal{CT}(G)
\end{equation}
defines the \emph{$n$-th $\varphi$-cotrajectory of $V$}.   \textit{The topological entropy} of $\varphi$ (in the sense of Hood) is 
 \begin{equation}\mathsf{h}_{\mathsf{top}}(\varphi)=\sup\left\{\limsup_{n\to\infty} \left(\frac{-\log\mu(C_n(\varphi,V))}{n}\right) \  \ \Big|   \ \ V\in\mathcal{CT}(G)\right\},
 \end{equation} %where \begin{equation}H_{top}(\varphi,V)=\limsup_{n\to\infty\frac{-\log\mu(C_n(\varphi,V))}{n}. \end{equation}
%and the \emph{topological entropy} of $\varphi$ is
Adler and others \cite{AdlerKonheimMcAndrew, B, P}  investigated the aforementioned notions, stressing on  dynamical properties of topological structures  with relations with ergodic theory and  mathematical physics.

Following \cite{DS, daf}, we may introduce the \emph{topological entropy of} a locally compact group $G$ as \begin{equation}\label{importantset}\mathsf{E}_{\mathsf{top}}(G)=\{\mathsf{h}_{\mathsf{top}}(\varphi)\mid \varphi\in \mathrm{End}(G)\},\end{equation}
where $\mathrm{End}(G)$ denotes the ring of continuous endomorphisms of $G$ and $\mathrm{Aut}(G)$ the group of continuous automorphisms of $G$. Here we  investigate  
%\begin{equation}
%\{\mathsf{E}_{\mathsf{top}}(G)\mid G\ \mbox{is  a compact group}\},
%\end{equation} 
the cardinality of  \eqref{importantset} and relations with structural properties, as  made in   \cite{B,  patrao2, daf, patrao1, Sch, SVirili, Y}. 

%Some relevant problems concerning discrete dynamical systems deal with the values of entropy; this is the case, for instance, of  the problem of  existence of topological automorphisms of compact groups with arbitrary small topological entropy. 
Denoting with $\widehat{ \mathbb{Q}}$ the topological dual (in the sense of Pontryagin) of the additive group $\mathbb{Q}$ of the rationals, we note that
\begin{equation}\label{minimizationproblem}\inf\{\mathsf{E}_{\mathsf{top}}(G)\setminus\{0\}\mid G\ \text{is a compact group}\}=\inf(\{\mathsf{h}_{\mathsf{top}}(\varphi) \mid \varphi \in \mathrm{Aut}(\widehat{\mathbb{Q}}^n),\ n\in\mathbb{N}\}\setminus\{0\})
\end{equation} and a formula of  Yuzvinski \cite{Y} shows that $\mathsf{h}_{\mathsf{top}}(\varphi)$ can be calculated from the solutions of the characteristic polynomial of $\varphi$ (see \cite{LW,Y}). Looking at locally compact groups, we also note  that  $\mathsf{h}_{\mathsf{top}}(\psi)$ is finite for any  $\psi \in \mathrm{End}(\mathbb{R})$. Actually, we can do much more: given  $t \in \mathsf{E}_{\mathsf{top}}(\mathbb{R}) \setminus \{+\infty\}$ we may construct  $\psi \in \mathrm{Aut}(\mathbb{R})$ of  $\mathsf{h}_{\mathsf{top}}(\psi)=t$, see \cite{B,  Walters}. 

Following \cite{DS, daf, BrunoVirili, SVirili}, we  introduce (for a locally compact group $G$)
\begin{equation}\label{classes}\mathfrak E_0=\{G\mid \mathsf{E}_{\mathsf{top}}(G)=\{0\}\}\quad \text{and}\quad \mathfrak E_{<\infty}=\{G\mid \mathsf{E}_{\mathsf{top}}(G)=[0, +\infty)\}
\end{equation} 
and note that there are results, which describe  the abelian cases in $\mathfrak E_{<\infty}$ and  $\mathfrak E_0$. The characterization of groups in $\mathfrak E_0$ can indicate the presence of structural theorems. For instance, finite abelian groups are in $\mathfrak E_0$ and  have a decomposition in direct product. On the other hand, very little is known in the nonabelian case in $\mathfrak E_{<\infty}$ and  $\mathfrak E_0$.

Following \cite[Definition 2.2]{HHR} and denoting by $\mathbb{P}$ the set of all primes, an element $g$ of a locally compact group $G$ is called $p$-$element$, if the sequence $g^{p^k}$ with $k \in \mathbb{N}$ tends to the identity element  in $G$. A locally compact group $G$ is called  $p$-$group$, if $G$ coincides with  
\begin{equation}
G_p=\{g \in G \mid  g  \ \mbox{is a} \ p\mbox{-element}\}=\{g \in G \mid \lim_{k \to \infty} g^{p^k}=1\}.
\end{equation}% Of course a \emph{locally compact $p$-group} is a locally compact group $G$ such that $G=G_p$. 
A maximal $p$-subgroup of a locally compact group $G$ is called $p$-$Sylow$ $subgroup$ of $G$. Note that $G_p$ turns out to be a closed subgroup by  \cite[Lemma 2.6]{HHR}, when $G$ is totally disconnected. Following \cite{HHR, hofmor}, we  denote by $G_0$ the connected component of the identity %of a locally compact group $G$ 
and  say that $G$ is \textit{compactly covered}, if  for an arbitrary $x \in G$ we can always find a compact subgroup $C$ of $G$ such that $x \in C$.  From \cite[p.5]{HHR}, a \textit{compact element} of $G$ is an element $g \in G$ such that $\overline{\langle g \rangle}$ is compact and the set 
\begin{equation}
\mathrm{comp}(G)=\{g \in G \mid g \ \mbox{is a compact element  } \}
\end{equation}
is described in \cite[Proposition 1.3, Lemma 1.6]{HHR}. For instance,  $G= \mathrm{comp}(G)$ when $G$ is locally compact abelian, but in general $\mathrm{comp}(G)$ is just a subset of $G$, not necessarily a subgroup.  Note that $\mathrm{comp}(G)$ is denoted by $B(G)$ in \cite{bd, DS, daf}; similarly, $G_0$  by $c(G)$.    Following  \cite[Proposition 1.3]{HHR}, we call  \emph{periodic} those locally compact groups $G$ such that $G_0=1$  and $\overline{\langle g \rangle}$ is compact for all $g \in G$. Of course, periodic  locally compact groups are totally disconnected, so their Sylow $p$-subgroups are closed and $\mathrm{comp}(G)=G$ by \cite[Lemma 1.6]{HHR}. %In a locally compact group $G$, one can also introduce for an element $g \in G$ the set  $\pi(g)$ of all primes such that $\overline{\langle g \rangle}$ has a nontrivial $p$-Sylow subgroup, putting $\pi(g) = 0$ when $\overline{\langle g \rangle} = \mathbb{Z}$ is the additive group of the integers. For any  $\sigma \subseteq \mathbb{P}$, we say that $g$ is a $\sigma$-$element$ of  a periodic locally compact group $G$, if $\pi(g)$ is a subset of $\sigma$. In this situation we may  consider\begin{equation}G_\sigma = \{g \in G \mid  g \ \mbox{is a} \ \sigma \mbox{-element}\}\end{equation} and say that a periodic locally compact group $G$ is a $\sigma$-$group$, if $\pi(G)=\bigcup_{g \in G} \pi(g) $ is contained in $\sigma$.  A maximal $\sigma$-subgroup of a periodic locally compact group $G$ is a $\sigma$-$Sylow$ $subgroup$ of $G$, finding in the special case of  $\sigma= \{p\}$ the  notion of $p$-Sylow subgroup. 

%According to \cite[Lemma 3.91]{HHR}, the $\mathbb{Z}$-$rank$ of an (discrete) abelian group $G$ is the \textit{torsion-free rank} of $G$, i.e. the dimension of $G \otimes \mathbb{Q}$ as $\mathbb{Q}$-module. When $G$ is torsion-free abelian,  the $\mathbb{Z}$-rank of $G$  is the dimension of its Pontryagin dual $\widehat{G}$. 
A locally compact group $G$ is {\em topologically finitely generated}, if there exists a finite subset $X$ of $G$ such that $G=\overline{\langle X \rangle}$. In particular, a locally compact $p$-group $G$ has \textit{ﬁnite p-rank}, if %every topologically ﬁnitely generated closed subgroup $H$ has a set $Y$ of topological generators of $|Y|=n$ and any other  topologically ﬁnitely generated subgroup of $G$ cannot be topologically generated withless than $n$ elements.  
\begin{equation}
\mathrm{rank}_p(G) = \max \{ \mathrm{rank}_p(H) \ | \ H \ \mbox{closed subgroup of} \ G\}
\end{equation}
is a positive integer, where also the following quantities are positive integers
\begin{equation}\label{defect}
\mathrm{rank}_p(H)=\min \{|Y| \ | \ Y \subseteq H \ \mbox{and} \ \overline{\langle Y \rangle} = H\}.
\end{equation}
For compact $p$-groups, see also  \cite[\S 2.4]{rz}. Following \cite{HHR, hofmor}, a locally compact  group $G$ is $compactly$ $generated$ if there exists a compact set $C$ such that $G=\langle C \rangle$. It is possible to provide examples of periodic locally compact  groups, which are not compactly generated.  It is also possible to provide examples which show that  ``topologically finitely generated groups'' and  ``compactly generated groups'' are two different notions. %Fortunately,  compactly generated locally compact abelian groups are well known :

\begin{theorem}[See, \cite{hofmor}, Theorem 7.57]\label{CompactlyGenerated}  Every compactly generated locally compact abelian group  is isomorphic to a direct sum $\mathbb{R}^d \oplus \mathbb{Z}^m \oplus K$ for a compact abelian group $K$ and two   nonnegative integers $d, m$. 
\end{theorem}

%In particular, Theorem \ref{CompactlyGenerated}  shows that we can  find a \textit{maximal compact open subgroup} $K(G)$ in any compactly generated locally compact abelian groups. Of course, $G/K(G)$ turns out to be a noncompact locally compact group which is compactly generated. 
We  are going to focus on specific classes of locally compact abelian groups and check whether the topological entropy of their continuous endomorphisms is finite or not;  results  of the type  of Theorem \ref{CompactlyGenerated} are fundamental for this scope.  Denote the cartesian sum of  countably many copies of $\mathbb{Z}$ by $\mathbb{Z}^{\mathbb{N}}=\{{(x_i)}_{i \in \mathbb{N}} \ | \  x_i \in \mathbb{Z}  \}$ and    by  $\mathbb{Z}^{(\mathbb{N})}=\{{(x_i)}_{i \in \mathbb{N}} \ | \  x_i \in \mathbb{Z} \ \mbox{and} \ x_i=0 \ \mbox{for almost all}  \ i  \}$ the direct sum of countably many copies of $\mathbb{Z}$. Denote by $\mathrm{dim}(A)$ the dimension of a compact abelian group $A$, that is, the dimension of the $\mathbb{Q}$-module $\mathbb{Q} \otimes \widehat{A}$ as per \cite[Definitions 8.23]{hofmor}. Note also from \cite[Corollary 7.58]{hofmor} that a connected compact abelian group $A$ of finite dimension is characterized to be the direct sum of finitely many copies of the torus $\mathbb{T}=\mathbb{R}/\mathbb{Z}$.

 %Following \cite[Appendix 1]{hofmor}, if $I$ is a given set of indices and  $A$ a topological abelian group,  $A^I=\{{(a_i)}_{i \in I} \ | \ a_i \in A \}$ denotes the cartesian sum of $A$ with index set in $I$, while $A^{(I)}=\{{(a_i)}_{i \in I} \ | \  a_i \in A \ \mbox{and} \ a_i=0 \ \forall i \in I \setminus F \}$ denotes the direct sum of $A$ with index set $I$, where $F$ is a finite nonempty subset of $I$. Both $A^I$ and $A^{(I)}$ are topological abelian groups, but $A^I$ contains $A^{(I)}$ in general. Note that one can also identify $A^I$ with the set of all functions from $I$ to $A$ endowed with the pointwise convergence, while  $A^{(I)}$ will be a subspace of $A^I$ of all definitively constant functions. Of course, we consider on $A^I$ the product topology, and on  $A^{(I)}$ the induced topology from $A^{I}$, since $A^{(I)}$ is always a subgroup of $A^{I}$. Let's focus on the following notion:

\begin{definition}[See \cite{Fuchs}, p.489]  A   (discrete) torsion-free abelian group $G$ is \textit{slender}, if for every homomorphism $ \alpha : {(e_{i})}_{i \in \mathbb{N}} \in \mathbb{Z}^{\mathbb{N}} \longmapsto  \alpha({(e_{i})}_{i \in \mathbb{N}})\in G$ we have  $\alpha({(e_{i})}_{i \in \mathbb{N}})=0$ for almost all $i$, where ${(e_{i})}_{i \in \mathbb{N}}$ is the sequence with  the $i$-th component equal to 1 and 0 elsewhere.
\end{definition}

 Our first main result can be now formulated:

\begin{theorem}[First Main Theorem]\label{olwethusproof}
  Let $G$ be a compactly generated locally compact abelian group. With the notations of Theorem \ref{CompactlyGenerated}, the following statements are  satisfied :
  \begin{itemize}
      \item[(a).]  If  $G$ is slender, then $G \in \mathfrak{E}_0$. Viceversa, if $G \in \mathfrak{E}_0$ and  $K=0$, then $G$ is slender. 
      \item[(b).] Assume that  $K$ is connected. Then $G \in \mathfrak{E}_{< \infty}$ if and only if $G \simeq  \mathbb{R}^d \oplus \mathbb{Z}^m \oplus \mathbb{T}^s$ for some nonnegative integers $d,m,s$.
  \end{itemize}
\end{theorem}
Note that computations of the topological entropy of continuous automorphisms (not  endomorphisms) of   $\mathbb{R}^d \oplus \mathbb{Z}^m \oplus \mathbb{T}^s$ are available in \cite[pp. 475--476]{P}. Also \cite{patrao2, patrao1} contain computations of the topological entropy of continuous endomorphims, but mostly of Lie groups. We go ahead and describe the finiteness of the topological entropy for some nonabelian locally compact groups, looking at the behaviour of the Sylow $p$-subgroups. This is our second main result.

\begin{theorem}[Second Main Theorem]\label{oldconj}
The continuous automorphisms of a  nilpotent periodic locally compact $p$-group $G$ have  finite topological entropy  whenever  $\mathrm{rank}_p(G)$ is finite. \end{theorem}

We can always find periodic locally compact  $p$-groups $G$ of  $\mathrm{rank}_p(G)=r$ in $\mathfrak E_{<\infty}$, looking at the direct sum $G=\mathbb{Z}^r_p$ of $r$ copies of the additive group of $p$-adic integers $\mathbb{Z}_p$. On the other hand, it is possible to find periodic locally compact $p$-groups of nilpotency class two and of finite $p$-rank, looking at  Heisenberg $p$-groups $\mathbb H_{n}(\mathbb{Q}_p)$ constructed with upper triangular $(n+2) \times (n+2)$  matrices with coefficients in the field of $p$-adic rationals $\mathbb{Q}_p$. These  are   neither abelian nor compact groups, and have finite topological entropy and finite $p$-rank large enough. %This is our third main result.

\begin{theorem}[Third Main Theorem]\label{Heisenbergpgroups}
The  Heisenberg group $\mathbb H_{n}(\mathbb{Q}_p)$ is a periodic locally compact nonabelian $p$-group of nilpotency class $2$ with $\mathrm{rank}_p( \mathbb H_{n}(\mathbb{Q}_p)) = 2n$, where $n$ is an arbitrary positive integer. Moreover $ \mathbb H_{n}(\mathbb{Q}_p)$ belongs to  $\mathfrak E_{<\infty}$, but not to $ \mathfrak E_0$.
\end{theorem}

Terminology and notations are standard and follow \cite{Fuchs, HHR, hofmor, lg, rz, rob}. After the statement of the main results in Section 1, the theory of slender groups is summarized  in Section 2 from \cite{Fuchs, rob} and some recent results on the finiteness of the topological entropy for periodic locally compact groups are summarized in Section 3 from \cite{AdlerKonheimMcAndrew, B, BrunoVirili, DS, daf, SVirili}. Section 4 is devoted to construct $\mathbb H_{n}(\mathbb{Q}_p)$ and to prove some results on the $p$-rank of these groups. Then we end with the proofs of Theorems \ref{olwethusproof}, \ref{oldconj} and \ref{Heisenbergpgroups} in Section 5.

\section{Previous Results on Slender Groups}\label{section11}

We recall properties of slender groups, originally noted by Nunke, Los and Sasiada, see \cite{Fuchs}.

\begin{lemma}[See \cite{Fuchs}, Chapter 13, \S 2] \label{SlenderGroupsProperties}\
\begin{itemize}
    \item[(i).] Subgroups of slender groups are slender;
\item[(ii).] Slender groups are torsion-free; 
\item[(iii).] $\mathbb{Q}$, $\mathbb{Z}_p$ and $\mathbb{Z}^{\mathbb{N}}$ are  not slender; 
\item[(iv).] A group which is slender cannot contain any subgroup isomorphic to $\mathbb{Q}$, $\mathbb{Z}_p$, or $\mathbb{Z}^{\mathbb{N}}$;
\item[(v).] Direct products of slender groups are slender. In particular,  $\mathbb{Z}^{(\mathbb{N})}$ is slender;
\item[(vi).] A torsion-free abelian group $G$ is slender if for every homomorpshim $f : \mathbb{Z}^{\mathbb{N}} \to G$ the  image $f(\mathbb{Z}^{\mathbb{N}})$ is a discrete finitely generated abelian group.
\end{itemize} 
\end{lemma}

From  \cite{HHR, hofmor, lg, rz}, we may consider a periodic  locally compact $p$-group $G$ (not necessarily abelian)  with $k$  positive integer and introduce the subgroups \begin{equation}\Omega_k(G)=\overline{\langle  g\mid g^{p^k}=1  \rangle} \ \ \mbox{and} \ \ \Omega^k(G)=\langle g^{p^k} \mid g \in G \rangle, \end{equation}  which  are fully invariant in $G$ and satisfy $G/\Omega_k(G)=\Omega^k(G)$. This allows us to introduce also \begin{equation}\mathrm{Div}(G)=\underset{k \in \mathbb{N}}{\bigcap}\Omega^k(G), \end{equation}  which turns out to be useful for various reasons. For instance,  if $G$ is an (discrete) abelian group (not necessarily a periodic locally compact $p$-group),  $\mathrm{Div}(G)$ as above  is still well defined and  we say that $G$ is $ divisible$, if $\mathrm{Div}(G) \supseteq G$, or  that $G$ is $reduced$, if the trivial subgroup of $G$ is the only divisible subgroup of $G$ (see  \cite[Appendix 1, Definition A1.29]{hofmor}). 

    \begin{lemma}[See \cite{hofmor}, Corollary 8.5]\label{Hofmann_Cor8.5} \
    
    For a compact abelian group $G$, the following conditions are  equivalent:
    \begin{itemize}
        \item[(i).] $G$ is totally disconnected;
        \item[(ii).] $\mathrm{Div}(G)=0$;
        \item[(iii).]  $\widehat{G}$   is a torsion group.
%        \item[(d).] $G$ has no nondegenerate one parameter subgroups;
%        \item[(e).] $G$ is totally arcwise disconnected.
    \end{itemize}
    \end{lemma}

Nunke and Sasiada  \cite[Chapter 13, \S 2]{Fuchs} %, we find useful to recall a classical argument here.  Assume that a torsion-free abelian group $G$ with $\vert G\vert<2^{\aleph_{0}}$ is slender and assume further that $G$ is divisible and define $\alpha:\mathbb{Z}^{\left(\mathbb{N}\right)}\rightarrow G$ with $\alpha(a_{i})=g$ for all $i\in\mathbb{N}$ with $0\neq g\in G$. If $\iota:\mathbb{Z}^{\left(\mathbb{N}\right)}\rightarrow \mathbb{Z}^{\mathbb{N}}$ is an embedding, then  the injectivity property of $G$ implies that $\iota$ extends to $\beta:\mathbb{Z}^{\mathbb{N}}\rightarrow G$ via the map $ \alpha(a_{i})=(\beta\iota)(a_{i})=\beta\left(\iota(a_{i})\right)=\beta(a_{i})\neq0$,    for all $i\in\mathbb{N}$. This is a contradiction with the slenderness of $G$. 
showed that slender groups cannot be divisible. The reader can refer also to \cite[Exercise 4.4.9]{rob} and their result is summarized below.
%This proves the first part of the  result  below (the remaining part is more elaborated to show).
    
    \begin{lemma}[See \cite{Fuchs}, Lemma 2.3, Sasiada's Theorem ] \label{slender_reduced}
    An abelian  group which is slender must be reduced. In addition,  if the group is countable, then the condition of being reduced is necessary and sufficient to conclude that the group is slender.
    \end{lemma} 
    
    Following the discussion in \cite[Chapter 1, \S 7]{Fuchs} and \cite[\S 1]{DeMarcoOrsatti1974}, we may consider an abelian group $G$ and a
   filter  $\mathcal{F}$ in the subgroups lattice $\mathrm{L}(G)$ of $G$. Automatically $\mathcal{F}$ defines a topology on $G$, if we declare $  \mathcal{B}=\{U  \mid U \in  \mathcal{F} \}$  
 to be a basis of open neighborhoods at the identity of $G$ 
and if for every $g \in G$ the cosets  $g\mathcal{B}=\{gU \mid U \in \mathcal{B}\}$
 form a basis of open neighborhoods at $g\in G$.  This topology is said to be  a \textit{linear topology} on $G$ (or more precisely a \textit{linear} $\mathcal{F}$-\textit{topology} on $G$). \textit{Linear groups} (in the sense of Orsatti and De Marco) are  abelian groups with linear topologies. A linear group $G$ is  \textit{complete}, if it is Hausdorff and every Cauchy net in $G$ has a limit in $G$. De Marco and Orsatti \cite{DeMarcoOrsatti1974} studied  Hausdorff linear groups: 
 
 \begin{definition}[See \cite{DeMarcoOrsatti1974}]\label{lomega}An abelian group $G$ belongs to the class $L\Omega$ if it admits a \textit{linear complete and nondiscrete, Hausdorff topology}. We say that $G$ belongs to the class $L\Omega_{1}$, if it belongs to $L\Omega$ and in addition its topology is metrizable. \end{definition}

 In fact the conditions of Definition \ref{lomega} are not verified simultaneously, that is, there are abelian linear groups which are not Hausdorff, or abelian linear groups which are not complete and so on. Of course,  abelian groups in $L\Omega_1$ are also in $L\Omega$, but examples can show that the viceversa is false.

   \begin{theorem}[See \cite{DeMarcoOrsatti1974}, Theorem 2.3]\label{lineartopologies}
     A torsion-free abelian group  possesses  a metrizable linear  complete nondiscrete topology if and only if it contains a copy of $\mathbb{Z}_{p},$ or of $\mathbb{Z}^{\mathbb{N}}$ as subgroup.
   \end{theorem}

%Of course, the additive group of $p$-adic integers in Theorem \ref{lineartopologies} should be endowed with its usual $p$-adic topology. Similarly, we intend a copy of $\mathbb{Z}^{\mathbb{N}}$ with its usual product topologyof the discrete topology on $\mathbb{Z}$. 
Note  that   all groups of $L\Omega_{1}$ are classified by Theorem \ref{lineartopologies}. Moreover Lemma \ref{SlenderGroupsProperties} shows that both $\mathbb{Z}_{p}$ and $\mathbb{Z}^{\mathbb{N}}$ are not slender, hence  Theorem \ref{lineartopologies} implies that  $G$ cannot be slender, if it is possible to endow $G$ of a metrizable linear complete nondiscrete topology. This  is reported  below: 
   
   \begin{theorem}[See \cite{DeMarcoOrsatti1974}, De Marco and Orsatti]\label{Characterization_of_SlenderGroups}
     Let $G$ be a reduced torsion-free abelian group. Then $G$ is slender if and only if $G$ does not belong to $ L\Omega_{1}$.
   \end{theorem}

Thanks to what we have seen until now :

\begin{lemma}\label{NoCompactSlenderGroups}
  There are no nontrivial compact abelian slender groups.
\end{lemma}
\begin{proof}
     Assume that $G$ is a compact abelian slender group. Lemma \ref{slender_reduced} along with Lemma \ref{Hofmann_Cor8.5} (a) and (b) imply that $G$ is totally disconnected. Then $G$ should be profinite by \cite[Theorem 1.34]{hofmor}, hence  projective limit of finite groups. Profinite abelian groups are not slender;  $\mathbb{Z}_{p}$ is a counterexample. From the contradiction,  there are no nontrivial compact slender groups.
\end{proof}

\section{Previous Results  for Periodic Locally Compact  Groups}\label{sechtop}

When we have  a totally disconnected locally compact group $G$, van Dantzig  \cite{vD} proved that 
\begin{equation}
    \mathcal{U}(G)=\{V\leq G\mid \text{$V$ compact and open}\}\end{equation}
is contained in $
\mathcal{CT}(G)$ and is  local basis. From \cite[Proposition 3.4 ]{daf}, we have that
\begin{equation}\label{htoptdlc}
\mathsf{h}_{\mathsf{top}}(\varphi)=\sup\Big\{\lim_{n\to\infty} \left(\frac{\log |V:C_{n}(\varphi,V)|}{n}\right) \ \ \Big| \ \ V \in \mathcal{U}(G) \Big\},
\end{equation}
where $C_{n}(\varphi,V)\in\mathcal{U}(G)$ and the index $|V:C_n(\varphi,V)|$ is finite. In fact, the set  $\mathsf{E}_{\mathsf{top}}(G)$ turns out to be   a countable subset of the real half-line in this situation.

Some relevant facts are reported below. The first regards discrete groups.

\begin{remark}[See \cite{daf}, Remark 2.4]\label{discrete=0} 
Discrete groups belong to $\mathfrak E_0$.
%In fact, $\mathcal{B}(G)=\{U\leq G\mid U\ \text{finite}\}$,  so if $\varphi\in\mathrm{End}(G)$ and $U\in\mathcal{B}(G)$, then $|U:C_n(\varphi,U)|\leq |U|$ for every $n\in\mathbb{N}_+$.Hence, $H_{top}(\varphi,U)=0$ and consequently $\mathsf{h}_{\mathsf{top}}(\varphi)=0$. 
\end{remark}

The second regards the additive group of $p$-adic integers.

\begin{corollary}[See \cite{daf}, Corollary 2.2]\label{invbase} 
Let $G$ be a locally compact group and $\varphi\in\mathrm{End}(G)$. If $\mathcal{S}\subseteq \mathcal{CT}(G)$ is a local basis of $G$ and $\mathcal{S}$ is realized by $\varphi$-invariant subgroups, then $\mathsf{h}_{\mathsf{top}}(\varphi)=0$. In particular, this applies to $\mathbb{Z}_p^n$, hence $\mathbb{Z}_p^n\in\mathfrak{E}_0$.
\end{corollary}

The computation of the topological entropy of continuous endomorphisms is somehow harder than that of continuous automorphisms, but we have  results for totally disconnected  groups.

\begin{corollary}[See \cite{daf}, Lemma 2.3, Theorem 3.11; See \cite{BrunoVirili},  Corollary 1.3]\label{mon} \

Let $G$ be a locally compact group and $\varphi\in\mathrm{End}(G)$. 
\begin{enumerate}
\item[{\rm (a)}.] If $H$ is a $\varphi$-invariant closed subgroup of $G$, then $\mathsf{h}_{\mathsf{top}}({\varphi_|}_H)\leq \mathsf{h}_{\mathsf{top}}(\varphi)$, and, if in addition $H$ is normal, then $\mathsf{h}_{\mathsf{top}}(\bar{\varphi}_{G/H})\leq \mathsf{h}_{\mathsf{top}}(\varphi)$, where $\bar{\varphi}_{G/H}:G/H\to G/H$ is induced by $\varphi$.
\item[{\rm (b)}.] If  $\mathcal{S} \subseteq \mathcal{U}(G)$ is a local basis  of $G$ such that   $\varphi^{-n}(V)$ is normal in $G$ for all $n$ and  $V \in \mathcal{S}$, then   $\mathsf{h}_{\mathsf{top}}(\varphi)=\mathsf{h}_{\mathsf{top}}(\bar \varphi_{G/\ker \varphi})$. 
\item[{\rm (c)}.] If $G$ is totally disconnected and $\varphi\in \mathrm{Aut}(G)$, then  $\mathsf{h}_{\mathsf{top}}(\varphi)=  \mathsf{h}_{\mathsf{top}}(\varphi_{|N}) + \mathsf{h}_{\mathsf{top}}(\bar \varphi_{G/N})$, where $N$ is a closed normal subgroup of $G$.

\end{enumerate}
\end{corollary}

The third regards $p$-adic rationals. Denoting the $p$-adic norm with $|-|_p$,  Yuzvinski's Formula  \cite{LW, Y} helps with the following computations:

\begin{theorem}[See \cite{LW}]\label{yuzvinskip} 
For $n\in\mathbb{N}$ and $\varphi \in \mathrm{End}
(\mathbb{Q}_p^n)$, we have
\begin{equation}\mathsf{h}_{\mathsf{top}}(\varphi)=\sum_{|\lambda_i|_p>1}\log|\lambda_i|_p,\end{equation}
where $\lambda_i$ (with $1\le i \le n$) is  eigenvalue of $\varphi$ in a finite extension of $\mathbb{Q}_p$. In particular,   $\mathbb{Q}_p^n\in\mathfrak{E}
_{<\infty}$.
\end{theorem}  

%From \cite{daf} we know that $\mathbb{Q}_p^n \in \mathfrak{E}_{<\infty}\setminus \mathfrak{E}_0$ and that $\mathbb{R}^n \in \mathfrak{E}_{<\infty} \setminus \mathfrak{E}_0$. 

Further criteria of finiteness are related to the notion of of finite $p$-rank.

\begin{theorem}[See \cite{HHR}, Theorem 3.97]\label{p-rank-finite} 
A locally compact abelian $p$-group $G$ has finite $p$-rank if and only $G \simeq \mathbb{Z}_p^\alpha \times\mathbb{Q}_p^\beta \times \mathbb{Z}(p^\infty)^\gamma \times E_p $ for some nonnegative integers $\alpha, \beta, \gamma, \delta$ and a finite $p$-group $E_p$ of $\mathrm{rank}_p(E_p)=\delta$. In particular, $G$ belongs to $\mathfrak{E}_{<\infty}$ and $\mathrm{rank}_p(G)=\alpha + \beta + \gamma + \delta.$  The case  of $G$ in $\mathfrak{E}_0$  is characterized by the condition $\beta=0$.
\end{theorem}

The above result  shows that the $p$-rank is preserved under Pontryagin duality. Indeed, \begin{equation}\widehat{G} ={(\mathbb{Z}_p^\alpha \times\mathbb{Q}_p^\beta \times \mathbb{Z}(p^\infty)^\gamma \times F_p)}^{\wedge} \cong \mathbb{Z}_p^\gamma \times\mathbb{Q}_p^\beta \times \mathbb{Z}(p^\infty)^\alpha \times F_p,
\end{equation} and so $\mathrm{rank}_p(\widehat{G}) = \mathrm{rank}_p(G)$. In particular, it can be seen   that $\mathbb{Q}_p^\beta \in \mathfrak{E}_{<\infty}\setminus \mathfrak{E}_0$,  $\mathbb{R}^d \in \mathfrak{E}_{<\infty} \setminus \mathfrak{E}_0$, $\mathbb{Z}^\gamma_p \in \mathfrak{E}_0$, $F_p \in \mathfrak{E}_0$ and ${\mathbb{Z}(p^\infty)}^\alpha \in \mathfrak{E}_0$.
%In particular, totally disconnected locally compact groups are zero dimensional, compact abelian connected Lie groups of finite dimension are isomorphic to the torus $\mathbb{T}^n$ and so are $n$-dimensional, while the direct product $\mathbb{T}^\infty$ of countably many copies of $\mathbb{T}$ is an infinite dimensional connected compact abelian Lie group. 
Further results are reported below in the abelian case.

\begin{theorem}[See \cite{daf}, Theorems 1.1 and 1.2]\label{origin}
Let $G$ be a locally compact abelian group. 
\begin{itemize}
\item[{\rm (i)}.] If $G$ belongs to $\mathfrak{E}_{<\infty}$, then its dimension should be finite;  
\item[{\rm (ii)}.]The viceversa of (i) above is true when $G$ is compact and $G/G_0$ belongs to $\mathfrak{E}_{<\infty}$;
\item[{\rm (iii)}.] If $G$ belongs to $\mathfrak{E}_0$, then $G$ is totally disconnected; moreover  a profinite group belongs to $ \mathfrak{E}_0$ if and only if it belongs to $ \mathfrak{E}_{<\infty}$;
  \item[{\rm (iv)}.] If $G$ is periodic,  then $G$ is in $ \mathfrak E_0$ if and only if all its $p$-Sylow subgroups $G_p$ do the same. 
%  \item[{\rm (v)}.] If $G$ is periodic, then the condition of  $G$ of belonging to  $\mathfrak E_{<\infty}$ is equivalent to the belonging of  $G_p$ to   $\mathfrak E_{<\infty}$  for all $p\in\mathbb{P}$. 
\end{itemize}
\end{theorem}

In the arguments which are used to prove Theorem \ref{origin}, the main logic is to find decompositions of the endomorphisms in portions where we can control the finiteness of the topological entropy. 
In fact we say that \emph{the Addition Theorem holds} for  $(G,\varphi,H)$ of a locally compact group $G$ with $\varphi\in\mathrm{End}(G)$ and a $\varphi$-invariant closed normal subgroup $N$ of $G$, if  
\begin{equation}\label{AT}
\mathsf{h}_{\mathsf{top}}(\varphi) =\mathsf{h}_{\mathsf{top}}({\varphi_|}_N)+\mathsf{h}_{\mathsf{top}}(\bar \varphi_{G/N}),
\end{equation}
or  briefly, we write that $AT(G,\varphi,N)$ \emph{holds}. Of course, \eqref{AT} is equivalent to the commutativity of the following diagram
\begin{equation}\label{classicalsituation}
\begin{CD}
0 @>>>N@>\iota>>G@>\pi>>G/N@>>>0\\
@. @V{\varphi_|}_NVV @V\varphi VV @V \bar \varphi_{G/N} VV\\
0 @>>>N@>\iota>>G@>\pi>>G/N@>>>0\\
\end{CD}
\end{equation}
Similarly,  \emph{$AT(G)$ holds} if $AT(G,\varphi,N)$, which is depending on $\varphi$ and $N$ in general,  is satisfied by all $\varphi$ and  $N$. From \cite[Proposition 3.6]{daf}, if $N$ is a fully invariant open subgroup of $G$ and $AT(N)$ holds, then also $AT(G)$ holds.

At this point it is important that we pause and look closely at the structure of compactly generated locally compact abelian groups of Theorem \ref{CompactlyGenerated}. First, we note that the groups that appear in the decomposition are either compact or totally disconnected, or isomorphic to $\mathbb{R}^d$ for some nonnegative integer $d$. Because of this observation, we record the following result:

\begin{lemma}[See \cite{daf}, Lemma 3.1]\label{ATT_for_CGLCA_Groups}
    Let $A$, $B$ be two locally compact groups that either are compact, or totally disconnected or isomorphic to $\mathbb{R}^d$ for some nonnegative integer $d$, and $f \in\mathrm{End}(A),$ $g \in\mathrm{End}(B)$. Consider $A \times B$ with the product topology and $f \times g \in\mathrm{End}(A \times B)$. Then \begin{equation}
        \mathsf{h}_{\mathsf{top}}(f \times g)=\mathsf{h}_{\mathsf{top}}(f)+\mathsf{h}_{\mathsf{top}}(g).
    \end{equation} 
\end{lemma}

Again the situation is very clear computationally for locally compact abelian groups.

\begin{theorem}[See \cite{daf}, Theorems 1.8 and 1.9]\label{LAAAAAst:thm}
Let $G$ be a totally disconnected locally compact abelian group. Then, for every $\varphi\in\mathrm{End}(G)$, we have
\begin{equation}\mathsf{h}_{\mathsf{top}}(\varphi)=\sum_{p\in \mathbb P} \mathsf{h}_{\mathsf{top}}({\varphi_|}_{G_{p}}).\end{equation}
If $G$ is also periodic, then $AT(G)$ holds if and only if $AT(G_p)$ holds  for all $p$-Sylow subgroups $G_p$.
%\begin{itemize}
%\item[{\rm (i)}.] If $AT(G_p)$ holds for all $p$-Sylow subgroups $G_p$ of $G$, then $AT(G)$ holds.
%\item[{\rm (ii)}.] In case $G$ is periodic, $AT(G)$ holds if and only if $AT(G_p)$ holds for every $p\in\mathbb{P}$.
%\item[{\rm (iii)}.] For  $n\in\mathbb{N}$, $AT(\mathbb{Q}_p^n)$ holds.
%\end{itemize}
\end{theorem}

Theorem \ref{LAAAAAst:thm} (ii) shows that Addition Theorems may be reduced to Addition Theorems on $p$-Sylow subgroups. This means that the presence of a decomposition  helps to determine groups  in $\mathfrak{E}_0$ or in $\mathfrak{E}_{<\infty}$, just looking at Sylow $p$-subgroups in $\mathfrak{E}_0$ or in $\mathfrak{E}_{<\infty}$. %In fact structural results help a lot in the computation of the finiteness of the topological entropy of endomorphisms and automorphisms. 

\begin{remark}%Finite groups are finitely generated, but finitely generated groups are not necessarily finite, for instance the infinite cyclic $\mathbb{Z}.$ 
%Compact abelian groups are compactly generated, but there are compactly generated abelian groups, which are not necessarily compact; i.e.: see $\mathbb{R}$ with the usual topology and $\mathbb{R} \in \mathfrak{E}_{<\infty}$. 
For  compactly generated locally compact abelian groups,   Theorem \ref{CompactlyGenerated} shows that Lemma \ref{ATT_for_CGLCA_Groups} can be applied and so we have an Addition Theorem. This helps to reduce the computation of the topological entropy of continuous endomorphisms  to the topological entropy of continuous endomorphisms arising from  factors.
\end{remark}

\section{Heisenberg groups on p-adic rationals}

As application of Corollary~\ref{invbase}, we have that a compact $p$-group  $G$ with local basis $\{\Omega^n(G)\mid n\in\mathbb{N}\}\subseteq \mathcal{U}(G)$  should belong to $\mathfrak{E}_0$. Note that this applies to $\mathbb{Z}_p^n\times F_p$, where  $F_p$ is  finite $p$-group. 

\begin{remark}Groups of the form $\mathbb{Z}_p \times F_p$ for $F_p$  finite nonabelian $p$-group are among the easiest examples of infinite nilpotent compact $p$-groups which can be produced in $ \mathfrak{E}_0$. Looking at \cite[Section 3.1]{lg},  a finite $p$-group $F_p$ is of maximal class if $|p^n|$ with $n>3$ and its nilpotency class is $c =n-1$.  Their costruction can be found in \cite[Examples 3.1.5]{lg}. Now $\mathbb{Z}_p \times F_p$ has nilpotency class exactly $n$ by Fitting's Lemma \cite[Lemma 1.1.21]{lg}. This means that we have already an example of an infinite nonabelian compact $p$-group of nilpotency class arbitrarily large in $\mathfrak{E}_0$. 
\end{remark}

Given a commutative unitary topological ring $R$, the \textit{Heisenberg group} on $R$ is the group of all $(n+2) \times (n+2)$-matrices of the following form
\begin{equation}\label{largeheisenberg1} M(A,B;c) = \left(\begin{array}{c|cccc|c} 1 & a_1 & a_2 & ... & a_n & c\\
\hline 0 & 1 & 0 & ... & 0 & b_1  \\ 
0 & 0 & 1 & ... & 0 & b_2  \\ 
... & ... & ... & ...  & ... &...  \\
0 & 0 & 0 & ... & 1 & b_n  \\ 
\hline 0 & 0 & 0 & 0 & 0 & 1\\ \end{array}\right)= 
\left(\begin{array}{cccc} 1 & A & c \\
O & I_n & B  \\ 
0 & O & 1  \\ 
\end{array}\right),
\end{equation}
where the block $O$ is of all zeros, $I_n$ denotes a identity matrix $n \times n$, $A$ the $n$-tuple row $(a_1, \ldots, a_n)$, $B$ the $n$-tuple column $(b_1, \ldots, b_n)$. Of course, for $n=1$ we get the usual representation of the Heisenberg group as group of matrices $3 \times 3$.

In particular,  the matrices  \eqref{largeheisenberg1} 
have coefficients $m_{ij}$ such that
\begin{equation}\label{largeheisenberg2}
m_{ij}=\left\{\begin{array}{lcl} 1,  \ \mbox{if} \ i=j,\\
0, \ \mbox{if} \ i>j, \ \mbox{or}  \ 1<i<j<n-1.\\
\end{array}\right.
 \end{equation}
Note that $\mathrm{GL}(R^{n+2})$ is the general linear group of dimension $n+2$ of all invertible matrices with coefficients in $R$, and the set of all matrices \eqref{largeheisenberg1} is denoted by $\mathbb{H}_{n}(R)$ and equipped with the product topology induced by the product topology in $R^{{(n+2)}^2}$. In particular, one can check that   $\mathbb{H}_{n}(R)$ is nilpotent of class $2$, since the center
\begin{equation}\label{largeheisenberg3}Z(\mathbb{H}_{n}(R))=\overline{[\mathbb{H}_{n}(R),\mathbb{H}_{n}(R)]}=\left\{\left(\begin{array}{cccc} 1 & O & c \\
O & I_n & O  \\ 
0 & O & 1  \\ 
\end{array}\right) \ \mid \ \ c \in R  \right\}
\end{equation}
is topologically isomorphic to $(R,+)$ and the central quotient
\begin{equation}\label{centralquotient}
\mathbb{H}_{n}(R)/Z(\mathbb{H}_{n}(R)) \cong \underbrace{ (R,+)\times (R,+) \times \ldots \times (R,+) }_{2n-\mbox{times}},
\end{equation} 
is topologically isomorphic to $2n$ copies of $(R,+)$. Note that for $R=\mathbb{Z}$, or $\mathbb{Z}_p$, or $\mathbb{Z}(p)$,\eqref{centralquotient} is topologically generated by the matrices of the following form
\begin{equation}\label{elementaryabelianquotient}  \left(\begin{array}{c|cccc|c} 1 & 1 & 0 & ... & 0 & 1\\
\hline 0 & 1 & 0 & ... & 0 & 0  \\ 
0 & 0 & 1 & ... & 0 & 0  \\ 
... & ... & ... & ...  & ... &...  \\
0 & 0 & 0 & ... & 1 & 0  \\ 
\hline 0 & 0 & 0 & 0 & 0 & 1\\ \end{array}\right), \ \ 
 \left(\begin{array}{c|cccc|c} 1 & 0 & 1 & ... & 0 & 1\\
\hline 0 & 1 & 0 & ... & 0 & 0  \\ 
0 & 0 & 1 & ... & 0 & 0  \\ 
... & ... & ... & ...  & ... &...  \\
0 & 0 & 0 & ... & 1 & 0  \\ 
\hline 0 & 0 & 0 & 0 & 0 & 1\\ \end{array}\right), \ \
\ldots \\ 
 \left(\begin{array}{c|cccc|c} 1 & 0 & 0 & ... & 1 & 1\\
\hline 0 & 1 & 0 & ... & 0 & 0  \\ 
0 & 0 & 1 & ... & 0 & 0  \\ 
... & ... & ... & ...  & ... &...  \\
0 & 0 & 0 & ... & 1 & 0  \\ 
\hline 0 & 0 & 0 & 0 & 0 & 1\\ \end{array}\right),
\end{equation}
along with the corresponding ones where the role of $A$ is played by $B$ in the last column
\begin{equation}\label{elementaryabelianquotientbis}  \left(\begin{array}{c|cccc|c} 1 & 0 & 0 & ... & 0 & 1\\
\hline 0 & 1 & 0 & ... & 0 & 1  \\ 
0 & 0 & 1 & ... & 0 & 0  \\ 
... & ... & ... & ...  & ... &...  \\
0 & 0 & 0 & ... & 1 & 0  \\ 
\hline 0 & 0 & 0 & 0 & 0 & 1\\ \end{array}\right), \ \ 
 \left(\begin{array}{c|cccc|c} 1 & 0 & 0 & ... & 0 & 1\\
\hline 0 & 1 & 0 & ... & 0 & 0  \\ 
0 & 0 & 1 & ... & 0 & 1  \\ 
... & ... & ... & ...  & ... &...  \\
0 & 0 & 0 & ... & 1 & 0  \\ 
\hline 0 & 0 & 0 & 0 & 0 & 1\\ \end{array}\right), \ \
\ldots \\ 
 \left(\begin{array}{c|cccc|c} 1 & 0 & 0 & ... & 0 & 1\\
\hline 0 & 1 & 0 & ... & 0 & 0  \\ 
0 & 0 & 1 & ... & 0 & 0  \\ 
... & ... & ... & ...  & ... &...  \\
0 & 0 & 0 & ... & 1 & 1  \\ 
\hline 0 & 0 & 0 & 0 & 0 & 1\\ \end{array}\right).
\end{equation}
In particular, we describe the nonabelian compact $p$-group $\mathbb{H}(\mathbb{Z}_p)$  below for $n=1$.

Note from \cite[Chapter 2]{rz} that the \textit{Frattini subgroup} $\mathrm{Frat}(G)$ of a profinite group $G$ is defined as the intersection of all its maximal open subgroups. Moreover it is a characteristic subgroup of $G$. An element $g$  of a profinite group $G$ is a \textit{nongenerator} if it can be omitted
from every generating set of $G$, that is, whenever $G = \overline{\langle X, g\rangle}$, then $G = \overline{\langle X \rangle}$. In particular, \begin{remark}\label{frattiniproperty}
We have that the set  of all nongenerators of  a profinite group $G$ coincides with  $\mathrm{Frat}(G)$, see \cite[Lemma 2.8.1]{rz}. Moreover \cite[Lemma 2.8.6]{rz} shows that  the minimal number of generators of a topologically finitely generated profinite group $G$ agrees with the minimal number of generators of $G/\mathrm{Frat}(G)$. For  compact $p$-group $G$, this means that $\mathrm{rank}_p(G)=\mathrm{rank}_p(G/\mathrm{Frat}(G))$.\end{remark}

\begin{example}\label{zpheisenbergsmall} For any prime $p$, consider a separated  bilinear  map \begin{equation} \omega : (x,y)  \in \mathbb{Z}_p \times  \mathbb{Z}_p \to \omega (x,y) \in \mathbb{Z}_p
\end{equation}  and the set $\mathbb{Z}_p^3$ endowed with the binary operation
\begin{equation} \label{h1}
\square  \ :  \ ((x_1, y_1, z_2), (x_2, y_2, z_2)) \in \mathbb{Z}_p^3 \times \mathbb{Z}_p^3  \mapsto (x_1, y_1, z_1) \ \square \ (x_2, y_2, z_2) \end{equation}
$$ =\ \Big(x_1 + x_2, \ \  y_1 + y_2,  \ \ z_1 + z_2
+ \omega (x_1,y_2)  \Big) \in \mathbb{Z}_p^3.$$
See terminology  in \cite[Definitions 2.1 and 2.2]{bd}.  In particular, $\square$ is not a commutative operation and $(\mathbb{Z}_p^3, \square)$ satisfies the algebraic axioms of group. Of course, the construction  depends on $\omega$ and $(\mathbb{Z}_p^3, \square)$ is topologically isomorphic to $\mathbb{H}(\mathbb{Z}_p)$ with the matrix product. 
The center, the Frattini subgroup and the derived subgroup of $\mathbb{H}(\mathbb{Z}_p)$ satisfy
\begin{equation}\label{h3}
\mathrm{Frat}(\mathbb{H}(\mathbb{Z}_p)) \supseteq Z(\mathbb{H}(\mathbb{Z}_p))=\{M(0,0;c) \ | \ c \in \mathbb{Z}_p\}=  \overline{[\mathbb{H}(\mathbb{Z}_p),\mathbb{H}(\mathbb{Z}_p)]} \simeq \mathbb{Z}_p.
\end{equation}

We can now look at topological generators  and relations for  $\mathbb{H}(\mathbb{Z}_p)$, finding that
\begin{equation}\label{h5}
\mathbb{H}(\mathbb{Z}_p) = \overline{\langle  M(1,0;0), M(0,1;0), M(0,0;1) \ | \ [M(1,0;0), M(0,1;0)]=M(0,0;1),  }
\end{equation}
\[\overline{[M(1,0;0), M(0,0;1)]=[M(0,1;0),M(0,0;1)]=I_3
 \rangle}.
\]
From Remark \ref{frattiniproperty}, the $p$-rank of $\mathbb{H}(\mathbb{Z}_p)$ can be reduced to the Frattini quotient, i.e.:
\begin{equation}\label{LAAAAAAAAAST}
\mathrm{rank}_p(\mathbb{H}(\mathbb{Z}_p))=\mathrm{rank}_p(\mathbb{H}(\mathbb{Z}_p)/\mathrm{Frat}(\mathbb{H}(\mathbb{Z}_p)))=2.
\end{equation} 
\end{example}

Indeed, $\mathrm{Frat}(G)=\overline{\Omega^1(G) [G,G]}$ for any compact $p$-group $G$
by \cite[Lemma 2.8.7 (c)]{rz}. Since this is true of course when $G=\mathbb{H}(\mathbb{Z}_p)$, 
in  \eqref{h3} one can compute $\mathrm{Frat}(\mathbb{H}(\mathbb{Z}_p))$ as follows.  Now $[\mathbb{H}(\mathbb{Z}_p),\mathbb{H}(\mathbb{Z}_p)] = Z(\mathbb{H}(\mathbb{Z}_p))$ and 
\begin{equation}
\Omega^1(\mathbb{H}(\mathbb{Z}_p)) = \left(\begin{array}{cccc} 1 & p\mathbb{Z}_p\ & \mathbb{Z}_p \\
0 & 1 & p\mathbb{Z}_p  \\ 
0 & 0 & 1  \\ 
\end{array}\right) 
\end{equation} 
are compact, 
so $\mathrm{Frat}(\mathbb{H}(\mathbb{Z}_p))=\Omega^1(\mathbb{H}(\mathbb{Z}_p)) Z(\mathbb{H}(\mathbb{Z}_p))$, since the subgroup $\Omega^1(\mathbb{H}(\mathbb{Z}_p)) Z(\mathbb{H}(\mathbb{Z}_p))$ is compact.
Therefore, $\mathbb{H}(\mathbb{Z}_p)/\mathrm{Frat}(\mathbb{H}(\mathbb{Z}_p)) \cong \mathbb{Z}(p)\times \mathbb{Z}(p)$. This proves the second equality on (\ref{LAAAAAAAAAST}).  Example \ref{zpheisenbergsmall} holds more generally than $R=\mathbb{Z}_p$, see  \cite[Theorem 2.5 and Lemma 5.5]{grundh} and  \cite[\S 4]{bd}. Now we look at  $\mathbb{H}_n(\mathbb{Z}_p)$ and  $\mathbb{H}_n(\mathbb{Q}_p)$ for $n$ large enough.

\begin{lemma}\label{prankqp}
The Heisenberg group $\mathbb{H}_{n}(\mathbb{Q}_p)$ is a locally compact nonabelian $p$-group of nilpotency class two and  \[\mathrm{rank}_p(\mathbb{H}_{n}(\mathbb{Q}_p))=\mathrm{rank}_p(\mathbb{H}_{n}(\mathbb{Q}_p)/Z(\mathbb{H}_{n}(\mathbb{Q}_p)))=2n.\] 
\end{lemma} 

\begin{proof} Looking at \eqref{largeheisenberg1},  \eqref{largeheisenberg2}, \eqref{largeheisenberg3} and \eqref{centralquotient}  with $R=\mathbb{Q}_p$, it is clear that  $\mathbb{H}_{n}(\mathbb{Q}_p)$ is a locally compact nonabelian $p$-group of nilpotency class two. Now consider \eqref{largeheisenberg1} and observe that
\begin{equation}\label{factorization}
H_1=Z(\mathbb{H}_{n}(\mathbb{Q}_p)) \times K_2 \simeq \mathbb{Q}_p \times \mathbb{Q}^n_p \ \ \mbox{and} \ \
  K_1 \simeq \mathbb{Q}^n_p.
  \end{equation} 
Moreover $H_1$ is a closed normal subgroup such that
\begin{equation}\label{factorizationbis}
\mathbb{H}_{n}(\mathbb{Q}_p) = H_1 \rtimes K_1
=\{h_1k_1 \mid h_1 \in H_1 \ \mbox{and} \ k_1 \in K_1 \}
\end{equation}
\[=\{zk_2k_1 \mid z \in Z(\mathbb{H}_{n}(\mathbb{Q}_p)) \  k_2 \in K_2,  k_1 \in K_1  \}=\{[u_2,u_1] k_2 k_1 \mid  k_2, u_2 \in  K_2 \ \mbox{and} \ k_1, u_1 \in K_1 \}, \]
because  we have \begin{equation}\label{h5add}Z(\mathbb{H}_n(\mathbb{Q}_p))=\overline{[\mathbb{H}_n(\mathbb{Q}_p),\mathbb{H}_n(\mathbb{Q}_p)]}=\overline{[K_2, K_1]}.
 \end{equation}  Therefore the $p$-rank of $\mathbb{H}_{n}(\mathbb{Q}_p)$ is reduced to that of $K_1 $ plus that of $K_2$, i.e.:  $2n$. \end{proof}

Lemma \ref{prankqp} can be proved with the idea of Example \ref{zpheisenbergsmall}, that is, noting that \begin{equation}\mathrm{Frat}(\mathbb H_{n}(\mathbb{Q}_p)) \supseteq Z(\mathbb H_{n}(\mathbb{Q}_p)),
 \end{equation}
and that the quotient $\mathbb H_{n}(\mathbb{Q}_p)/\mathrm{Frat}(\mathbb H_{n}(\mathbb{Q}_p))$ has $p$-rank $2n$, but we gave an argument based on the structure of semidirect product for Heisenberg groups. Moreover also here one could argue that  $\mathrm{Frat}(\mathbb{H}_n(\mathbb{Q}_p))=\overline{\Omega^1(\mathbb{H}_n(\mathbb{Q}_p)) [\mathbb{H}_n(\mathbb{Q}_p),\mathbb{H}_n(\mathbb{Q}_p)]}$, even if in this sitation we don't have a compact $p$-group but a periodic locally compact $p$-group.

\begin{remark}
  Looking at \cite[Lemma 2.4 and Theorem 2.5]{bd}, one can show that $\mathbb{H}(\mathbb{Z}_p)$ possesses abelian maximal subgroups of the following form
\begin{equation}\label{h5bis}
H_1 = Z(\mathbb{H}(\mathbb{Z}_p)) \oplus \overline{\langle M(1,0;0) \rangle} \simeq {\mathbb{Z}_p}^2 \ \ \mbox{and} \ \ H_2 = Z(\mathbb{H}(\mathbb{Z}_p)) \oplus \overline{\langle M(0,1;0) \rangle} \simeq {\mathbb{Z}_p}^2
\end{equation}
satisfying the following conditions: 
\begin{equation}\label{h5tris}
H_1 \cap H_2 = Z(\mathbb{H}(\mathbb{Z}_p)), \ \ H_1 \cap \overline{\langle M(0,1;0) \rangle} = 1, \ \ H_2 \cap \overline{\langle M(1,0;0) \rangle} = 1, 
\end{equation}
 \begin{equation}\label{h6}
\mathbb{H}(\mathbb{Z}_p)= H_1 \rtimes \overline{\langle M(0,1;0)} \rangle =H_2 \rtimes \overline{\langle M(1,0;0) \rangle} \simeq  {\mathbb{Z}_p}^2 \rtimes \mathbb{Z}_p.
\end{equation}
In particular \eqref{h6} shows that any element of $\mathbb{H}(\mathbb{Z}_p)$ can be written uniquely as product of an element of $H_1$ and of one of $\overline{\langle M(0,1;0) \rangle}=K_1$,  but any element of $H_1$ can be also written uniquely as product of an element of  $Z(\mathbb{H}(\mathbb{Z}_p))$ and of one of $\overline{\langle M(1,0;0) \rangle}=K_2$ by \eqref{h5bis}. 
In Fig.~1 we identify the aforementioned subgroups in the lattice of closed subgroups $\mathcal{SUB}(\mathbb{H}(\mathbb{Z}_p))$ of $\mathbb{H}(\mathbb{Z}_p)$. At the first level (beginning from the bottom of Fig.~1) we find the trivial subgroup. At the second level there are three subgroups isomorphic to $\mathbb{Z}_p$. At the third level there are  two subgroups isomorphic to the additive group ${\mathbb{Z}_p}^2$. At the fourth level we find the entire group. Note that  Fig.~1 shows only the subgroups that can be directly deduced from \eqref{h6} and not all the subgroups of $\mathbb{H}(\mathbb{Z}_p)$.
\begin{center}
\begin{tikzpicture}
%[scale=.8,auto=left,every node/.style={circle,fill=blue!10}]
[scale=.8,auto=left]
   \node (d) at (4,3) {$\mathbb{H}(\mathbb{Z}_p)$};
  \node (b) at (-0.5,2) {$H_1$};
  \node (e) at (-2,-0.5) {$K_2 $};
  \node (a) at (8,2) {$H_2$};
  \node
   (f) at (10,-0.5) {$K_1 $};
  \node (c) at (4,-0.5) {$Z(\mathbb{H}(\mathbb{Z}_p))$};
%  \node (g) at (4,-5) {$[\mathbb{H}(\mathbb{F}),\mathbb{H}(\mathbb{F})]$};
  \node (h) at (4,-4) {$1$};

%  \node (Fig) at (-10,2) {$\Gamma(S_3)$};
  
  %\draw (a) -- (b) -- (c)--(a);
  \draw (d) -- (a);
  \draw (d) -- (b);
%  \draw (d) -- (c);
  \draw (e) -- (b);
  \draw (f) -- (a);
%  \draw (g) -- (c);
  \draw (h) -- (c);
  \draw (e) -- (h) ;
  \draw (f) -- (h) ;
\draw (a) -- (c) ;
\draw (b) -- (c) ;

 \end{tikzpicture}
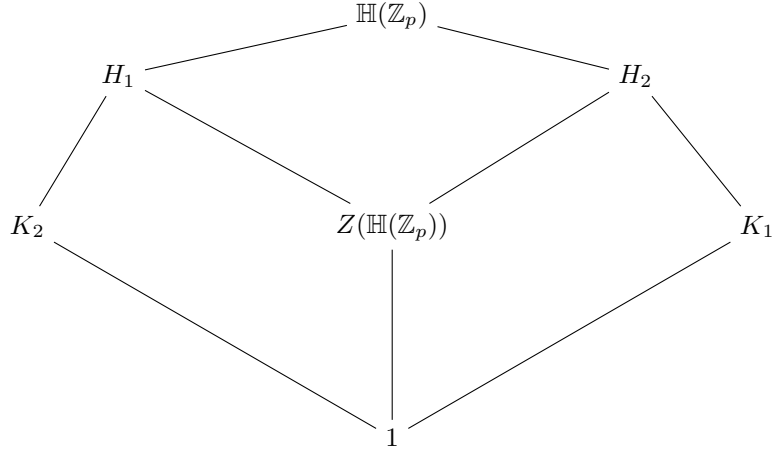
\centerline{\textbf{Figure 1}:  Some relevant subgroups in $\mathcal{SUB}(\mathbb{H}(\mathbb{Z}_p))$. }
 \end{center}
%\end{center}
\vspace{0.5em} 
 In fact one can see that, given the cardinality of the continuum $\mathfrak{c}$ and
fixed $\xi \in \mathbb{Z}_p$, the subset $M_{\xi}$ of all matrices $M (a, \xi a; t) \in \mathbb{H}(\mathbb{Z}_p)$ is a  maximal abelian subgroup of $\mathbb{H}(\mathbb{Z}_p)$, and of course there are $\mathfrak{c}$  of this type.

\end{remark}

\section{Proofs of  Main Theorems}

\begin{proof}[Proof of Theorem \ref{olwethusproof}]

     (a). If $G$ is a compactly generated locally compact abelian group, then Theorem \ref{CompactlyGenerated} implies  $G\cong\mathbb{R}^d\oplus\mathbb{Z}^{m}\oplus K$ for a compact abelian group $K$ and nonnegative integers $m,d$.  Assume in addition that $G$ is slender. Lemma \ref{SlenderGroupsProperties} shows that subgroups of slender groups are slender. Then Lemma \ref{slender_reduced} implies  $n=0$, that is,  $G \simeq \mathbb{Z}^{m}\oplus K$.   Lemma \ref{NoCompactSlenderGroups} implies  $K=0$. Hence $G \simeq \mathbb{Z}^{m}$, and since $\mathbb{Z}^{m}\in\mathfrak{E}_{0}$, the first part of the result follows.   Assume now that $G\in\mathfrak{E}_{0}$ and that $G \cong \mathbb{R}^d \oplus  \mathbb{Z}^m  $.  Since $\mathbb{R}^d \in \mathfrak{E}_{\infty} \setminus \mathfrak{E}_0$, $G$ should be totally disconnected by Theorem \ref{origin} (iii) and so $G \simeq \mathbb{Z}^m$ which is slender. The result follows completely.
         
         (b). From   Theorem \ref{CompactlyGenerated} and the assumption that $K$ is a connected compact abelian group, we have that $G \simeq  \mathbb{R}^d \oplus\mathbb{Z}^{m}\oplus K$ with $K$ of $\mathrm{dim} (K)$ eventually infinite. Then         
         \begin{equation}\label{dim} \mathrm{dim}(G)=  \mathrm{dim} (\mathbb{R}^d) + \mathrm{dim}(\mathbb{Z}^{m}) +  \mathrm{dim}(K)=d + 0 +  \mathrm{dim}(K)
         \end{equation} and this shows that
$\mathrm{dim}(G) < \infty $ if and only if   $\mathrm{dim}(K) < \infty $ if and only if  $K=\mathbb{T}^s$ for some nonnegative integer $s$, see \cite[Corollary 8.22 (5)]{hofmor}.    From Theorem \ref{origin} (i), this means that if $G \in \mathfrak{E}_{<\infty}$, then $\mathrm{dim}(G)< \infty$ hence $\mathrm{dim}(K) < \infty$, and so $G \simeq \mathbb{R}^d \oplus\mathbb{Z}^{m}\oplus \mathbb{T}^s$.  Conversely, assume that $G \simeq \mathbb{R}^d \oplus\mathbb{Z}^{m}\oplus \mathbb{T}^s$.   We may apply Lemma \ref{ATT_for_CGLCA_Groups} with summands $\mathbb{R}^d \in \mathfrak{E}_{<\infty}$,  $\mathbb{Z}^{m} \in \mathfrak{E}_0$ and $\mathbb{T}^s  \in \mathfrak{E}_{<\infty}$, concluding $G \in \mathfrak{E}_{<\infty}$. Note that the computations of topological entropy, which allows us to have $\mathbb{R}^d \in \mathfrak{E}_{<\infty}$,  $\mathbb{Z}^{m} \in \mathfrak{E}_0$ and $\mathbb{T}^s  \in \mathfrak{E}_{<\infty}$, are well known, see \cite{B, LW, SVirili}. The result follows.
     
\end{proof}

\begin{proof}[Proof of Theorem \ref{oldconj}]
First assume that $G$ has $\mathrm{rank}_p(G)<\infty$. We note that closed subgroups and quotients of $G$ are again periodic locally compact $p$-groups. The topological lower central series of $G$ of length $c$ has closed characteristic $p$-subgroups  $ \overline{\gamma_i(G)}$  (with $i=1,2,\ldots, c$) such that \begin{equation}G=\overline{\gamma_1(G)} \ge \overline{\gamma_2(G)}=\overline{[G,G]} \ge \overline{\gamma_3(G)}=\overline{[[G,G],G]} \ge \ldots \ge \overline{\gamma_c(G)} \ge \overline{\gamma_{c+1}(G)}=1\end{equation} and  $\overline{\gamma_i(G)}/\overline{\gamma_{i+1}(G)}$ are locally compact abelian $p$-groups for all $i$. Note also that  closed subgroups and quotients of a periodic locally compact $p$-group of finite $p$-rank have finite $p$-rank. This means that if $G$ has finite $p$-rank, then  $\overline{\gamma_i(G)}/\overline{\gamma_{i+1}(G)}$ are of the form of those in Theorem \ref{p-rank-finite}, and in particular continuous automorphisms of $\overline{\gamma_i(G)}/\overline{\gamma_{i+1}(G)} $ have finite topological entropy. Now we do induction on $c$. Assume $c=1$. Then $G$ is a locally compact abelian group of finite $p$-rank and the result is true by Theorem \ref{p-rank-finite}, because in this situation the continuous automorphisms of $G$ should have finite topological entropy. Assume $c>1$ and that the result is true for all  periodic nilpotent locally compact $p$-groups of derived length at  most $c-1$. Then the continuous automorphisms of $N=\overline{\gamma_c(G)}$  have finite topological entropy, since $N$  is abelian, but also those of $G/N$ have finite topological entropy, since $G/N$ is a locally compact abelian $p$-group of finite $p$-rank. From  Addition Theorem  for continuous automorphisms of totally disconnected locally compact abelian groups (see \cite[Addition Theorem 10]{P}, or Corollary \ref{mon} (b)) we conclude that $AT(G, \varphi, N)$ holds for every continuous automorphism $\varphi$ of $G$. The result follows.  %Viceversa assume that each automorphism of $G$ has finite topological entropy....
\end{proof}

\begin{proof}[Proof of Theorem \ref{Heisenbergpgroups}] From Lemma \ref{prankqp}, the Heisenberg group $\mathbb{H}_{n}(\mathbb{Q}_p)$ is a periodic locally compact nonabelian $p$-group of nilpotency class two and  $\mathrm{rank}_p(\mathbb{H}_{n}(\mathbb{Q}_p))=2n$. Then we shall only prove that   $ \mathbb H_{n}(\mathbb{Q}_p)$ belongs to  $\mathfrak E_{<\infty}$, but not to $ \mathfrak E_0$.

Assume that $n=1$. From \cite[Theorem 6.8]{daf} we know that $\mathbb{H}(\mathbb{Q}_p)$ belongs to  $\mathfrak E_{<\infty}$, but not to $ \mathfrak E_0$. Then there exists a subgroup $S$ of $\mathbb{H}_n(\mathbb{Q}_p)$ which is isomorphic to $\mathbb{H}(\mathbb{Q}_p)$ as topological group, for instance $S$ can be realized  putting in \eqref{largeheisenberg1} the condition $a_i=b_i=0$ for all $i=2,3,\ldots,n$. This is sufficient to show that $\mathbb{H}_n(\mathbb{Q}_p)$ cannot be in $\mathfrak E_0$, since it contains a subgroup $S$ which is not in $\mathfrak E_0$. It remains to check that $\mathbb{H}_n(\mathbb{Q}_p)$ belongs to   $\mathfrak E_{<\infty}$ and we adapt the argument of \cite[Proof of Theorem 6.8]{daf} for this scope.

Consider $\varphi \in \mathrm{End}(\mathbb{H}_n(\mathbb{Q}_p))$ and $N=\ker \varphi$; we claim that $\mathsf{h}_{\mathsf{top}}(\varphi)<\infty$. 

Assume that $N=1$. We claim that $\varphi \in \mathrm{Aut}(\mathbb{H}_n(\mathbb{Q}_p))$. Since $Z(\mathbb{H}_n(\mathbb{Q}_p))$ is fully invariant, $\varphi_{|Z(\mathbb{H}_n(\mathbb{Q}_p))}$ is  injective, hence $\varphi_{|Z(\mathbb{H}_n(\mathbb{Q}_p))}  $ is a continuous automorphism of $Z(\mathbb{H}_n(\mathbb{Q}_p))$. In particular, $\varphi^{-1}(Z(\mathbb{H}_n(\mathbb{Q}_p))) = Z(\mathbb{H}_n(\mathbb{Q}_p))$ and so   $\bar {\varphi}_{|\mathbb{H}_n(\mathbb{Q}_p)/Z(\mathbb{H}_n(\mathbb{Q}_p))}$ on  $\mathbb{H}_n(\mathbb{Q}_p)/Z(\mathbb{H}_n(\mathbb{Q}_p))$ is  injective. In fact  it  is  a continuous automorphism of $\mathbb{H}_n(\mathbb{Q}_p)/Z(\mathbb{H}_n(\mathbb{Q}_p)) \simeq \mathbb{Q}^{2n}_p$.  Now   $\mathbb{H}_n(\mathbb{Q}_p)$ is a totally disconnected locally compact group, which can be also realized as union of countably many compact sets, and so $\varphi$ is a continuous automorphism by the Open Mapping Theorem \cite[Appendix 1, Exercise EA1.21]{hofmor}.  We may apply Addition Theorems on closed normal subgroups for continuous automorphisms of locally compact groups as per Corollary  \ref{mon} (c), concluding  $\mathsf{h}_{\mathsf{top}}(\varphi)<\infty$ from the fact that both  $\mathsf{h}_{\mathsf{top}}(\varphi_{|Z(\mathbb{H}_n(\mathbb{Q}_p))})<\infty$ and $\mathsf{h}_{\mathsf{top}}(\bar{ \varphi}_{\mathbb{H}_n(\mathbb{Q}_p)/Z(\mathbb{H}_n(\mathbb{Q}_p))})< \infty$ by Theorem \ref{p-rank-finite}.  

Now assume that  $N=\ker \varphi \neq 1$. First we show that $N\cap Z(\mathbb{H}_n(\mathbb{Q}_p))$ is nontrivial and then that $Z(\mathbb{H}_n(\mathbb{Q}_p))\subseteq N$.  If there exists some $y\in N \setminus Z(\mathbb{H}_n(\mathbb{Q}_p))$, then there exists $x\in \mathbb{H}_n(\mathbb{Q}_p)$ such that  $[x,y]$ is nontrivial. This implies that $N\cap [\mathbb{H}_n(\mathbb{Q}_p),\mathbb{H}_n(\mathbb{Q}_p)]$ is nontrivial, because $[x,y]\in N$. The claim follows and  $N\cap Z(\mathbb{H}_n(\mathbb{Q}_p))$ is a nontrivial closed subgroup of $Z(\mathbb{H}_n(\mathbb{Q}_p))$, hence $Z(\mathbb{H}_n(\mathbb{Q}_p))/(N\cap Z(\mathbb{H}_n(\mathbb{Q}_p)))$ is torsion because nontrivial quotient of $\mathbb{Q}_p$. On the other hand, $Z(\mathbb{H}_n(\mathbb{Q}_p))/(N\cap Z(\mathbb{H}_n(\mathbb{Q}_p))) \cong \varphi(Z(\mathbb{H}_n(\mathbb{Q}_p))) $ is  a subgroup of $\mathbb{H}_n(\mathbb{Q}_p)$ (up to continuous isomorphisms), hence torsion-free. Consequently $Z(\mathbb{H}_n(\mathbb{Q}_p))/(N\cap Z(\mathbb{H}_n(\mathbb{Q}_p)))$  is trivial, and the other claim  $Z(\mathbb{H}_n(\mathbb{Q}_p)) \subseteq N$ follows. Since $N$ contains $Z(\mathbb{H}_n(\mathbb{Q}_p))=[\mathbb{H}_n(\mathbb{Q}_p),\mathbb{H}_n(\mathbb{Q}_p)]$, we may apply Addition Theorems as per Corollary  \ref{mon} (b), hence  $\mathsf{h}_{\mathsf{top}}(\varphi)=    \mathsf{h}_{\mathsf{top}}(\bar \varphi_{\mathbb{H}_n(\mathbb{Q}_p)/N})$ is finite by Theorem~\ref{p-rank-finite}. Therefore the result follows.

\end{proof}

%\subsection*{Acknowledgements} We thank Prof. D. Dikranjan for comments on the original version of the present manuscript. We also thank National Research Foundation of South Africa (South Africa) and Gruppo Nazionale per la Fisica Matematica of INdAM (Italy).

\end{document}